\documentclass[12pt]{article}

\usepackage{amsfonts}
\usepackage{graphicx}
\usepackage{amsthm}
\usepackage{amssymb}
\usepackage{amsmath}

\newtheorem{theorem}{Theorem}[section]
\newtheorem{proposition}{Proposition}[section]
\newtheorem{corollary}{Corollary}[section]
\newtheorem{lemma}{Lemma}[section]
\newtheorem{definition}{Definition}[section]

\newtheorem{remark}{Remark}[section]

\newtheorem{example}{Example}[]

\begin{document}

\title{Periodic projections of alternating knots}
\author{Antonio F. Costa and Cam Van Quach-Hongler}
\maketitle

\begin{abstract}
\noindent This paper is devoted to prove the existence of $q$-periodic
alternating projections of prime alternating $q$-periodic knots. The main
tool is the Menasco-Thistlethwaite's Flyping theorem.

\noindent Let $K$ be an oriented prime alternating knot that is $q$-periodic
with $q\geq 3$, i.e. $K$ admits a symmetry that is a rotation of order $q$.
Then $K$ has an alternating $q$-periodic projection.\newline
\noindent \noindent As applications, we obtain the crossing number of a $q$%
-periodic alternating knot with $q\geq 3$ is a multiple of $q$ and we give an elementary proof that the knot $%
12_{a634} $ is not 3-periodic; this proof does not depend on computer
computations as in \cite{jana}.
\end{abstract}

\section{Introduction}

In this paper, links (knots are one-component links) in $S^3$ and
projections in $S^2$ are assumed, unless otherwise indicated, prime and
oriented.\newline
The purpose of this paper is the study of the visibility of the periodicity
of alternating knots on alternating projections initiated in \cite{co}. The
main result is:

\textbf{Visibility Theorem 3.1} \textit{Let $K$ be an oriented prime
alternating knot that is $q$-periodic with $q\geq 3$. Then $K$ has a $q$%
-periodic alternating projection}.

We also obtain the following applications:

1. If Seifert's algorithm is applied on a $q$-periodic projection of an
oriented link, the resulting surface exhibits a $q$-periodic symmetry. Such
a surface is called \textbf{$q$-equivariant}. The topological types of
periodic homeomorphisms of bordered surfaces that are equivariant Seifert
surfaces of periodic links are studied in \cite{co1}. A. Edmonds \cite{ed}
shows that if a knot $K$ is of period $q$, then there is a $q$-equivariant
Seifert surface for $K$, which has the genus of $K$. For $K$ a $q$-periodic
prime oriented alternating knot with $q\geq 3$, the strategy explained in
the proof of Theorem 3.1 enables to exhibit the realization of a $q$%
-equivariant surface from Seifert's algorithm which has the genus of $K$.

2. The following result is also a direct consequence of Theorem 3.1:

\textbf{Proposition 3.4} (Conjecture in Section 1.4 of \cite{co}) The crossing
number of a $q$-periodic alternating knot with $q\geq 3$ is a multiple of $q$.

3. The Murasugi decomposition into atoms of $12a_{634}$ gives rise to an
adjacency graph which is a tree of 2-vertices (\cite{quwe1}, \cite{quwe2}).
According to Visibility Theorem 3.1 and Lemma 3.2 (which is an application
of Corollary 1 in \cite{co}), we deduce that $12a_{634}$ is not 3-periodic.
We thank C. Livingstone to point out the existence of a computer proof of
this fact by S. Jabuka and S. Naik \cite{jana}.

\subsection{Organization of the paper}

For the study of the visibility of the periodicity of the alternating knots
on alternating projections, we will call upon the \textbf{canonical
decomposition of link projections} recalled in \S 2, as it was done for the
visibility of achirality of the alternating knots in \cite{erquwe2}. The
decomposition of a link projection $\Pi $ is carried out by a family of 
\textbf{canonical Conway circles} which decomposes $(S^{2},\Pi )$ into 
\textbf{diagrams} called \textbf{jewels} and \textbf{twisted band diagrams};
the \textbf{arborescent part} of $\Pi $ is the union of the twisted band
diagrams of $\Pi $. The decomposition of the diagram $(S^{2},\Pi )$ by the
canonical Conway circles is a 2-dimensional version of the decomposition of
Bonahon-Siebenman \cite{bosi} of $(S^{3},K)$ into an \textbf{algebraic part} 
$(A,A\cap K)$ and a \textbf{non-algebraic part} $(N,N\cap K)$. Each
component of $\partial A=\partial N$ is a 2-sphere that cuts $K$ in four
points and is called a \textbf{Conway sphere}. We now assume that links and
projections we consider are alternating. In our terminology, the
2-dimensional notion \textquotedblleft arborescent" implies the
3-dimensional notion \textquotedblleft algebraic" (see the definition for
instance in \cite{thi}). The projection of a Conway sphere on $S^{2}$ is a
Conway circle. The inverse is not true: there are ``hidden"
Conway spheres that do not project on Conway circles on alternating
projections (\cite{thi}). However since our point of view is strictly
2-dimensional and based only on alternating projections, this case does not
affect us.\newline
The notion of \textbf{flype} in alternating projections (Fig. 7) is at the
heart of our analysis and lies completely in their \textbf{arborescent part}%
. According to Menasco-Thistlethwaite's Flyping theorem \cite{meth}, two
reduced alternating projections $\Pi _{1}$ and $\Pi _{2}$ of an isotopy
class of an alternating link $K$ are related by a finite sequence of flypes,
up to homeomorphisms of $S^{2}$ onto itself. 
Starting from the canonical decomposition of a projection of $\Pi (K)$ of $K$%
, we associate \textbf{canonical and essential structure trees} (as recalled in \S 2) that
does not depend on the choice of an alternating projection. The canonical
and essential structure trees are invariants of the isotopy class of
alternating knots. For example, for rational links, their canonical
structure tree is a linear tree with integer-weighted vertices and their
essential structure tree is reduced to a vertex of rational weight. 

In \S 3 we study how the $q$-periodicity acts on the essential Conway
circles and on the diagrams of any alternating projection. With the help of
Kerekjarto's theorem \cite{coko} and Flyping Theorem of
Menasco-Thistlethwaite we prove Theorem 3.1 and we obtain a $q$-periodic
alternating projection by adjustments with flypes on any alternating prime $%
q $-periodic knot. With the help of the Murasugi decomposition into atoms for periodic
alternating knots and a result in \cite{co} which links the $q$-periodicity of
an alternating knot to the $q$-periodicity of its atoms, we finally show
that $12a_{634}$ is not 3-periodic. 

\section{Canonical Decomposition of a Projection}

\noindent In this section we do not assume that link projections are
alternating. A \textbf{projection} on $S^2$ is the image of a link in $S^3$
by a generic projection onto $S^2$, hence a labeled graph with $n$ 4-valent
crossing-vertices labeled to reflect under and over crossings.\\
In this paper the term ``{ \bf diagram}'' will be used to refer to a different object (see below 
\S 2.1).

\subsection{Diagrams}

Let $\Sigma$ be a compact connected planar surface embedded on the
projection sphere $S^2$. We denote by $k+1$ the number of connected components
of its boundary $\partial \Sigma$.

\begin{definition}
The pair $D = (\Sigma , \Gamma= \Pi \cap \Sigma)$ where $\Pi$ is a link
projection is called a \textbf{diagram} if for each connected component $C$
of $\partial \Sigma$, $C \cap \Pi$ is composed exactly of 4 points.
\end{definition}

\begin{remark}
A (link) projection $\Pi$ on $S^2$ is a diagram $(\Sigma, \Pi)$ where $\Sigma= S^2$.
\end{remark}

\begin{definition}
(1) A \textbf{trivial diagram} is a diagram homeomorphic to $\mathrm{T
([\infty])}$ (Fig. $1(a)$).\newline
(2) A \textbf{singleton} is a diagram homeomorphic to Fig. 1(b).
\end{definition}

\begin{figure}[h!]
\centering
\par
\includegraphics[scale=.4]{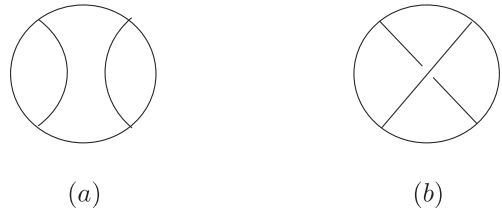}
\caption{$(a)$ Tangle $\mathrm{T ([\infty])}$ \quad \quad $(b)$ A singleton }
\end{figure}

\begin{definition}
A \textbf{Haseman circle} of a diagram $D = (\Sigma, \Gamma)$ is a circle $%
\gamma \subset \Sigma$ that intersects the projection $\Pi$ exactly in 4
points away from crossing points. A Haseman circle is said to be \textbf{%
compressible} if $\gamma$ bounds a disc $\Delta$ in $\Sigma$ such that $%
(\Delta,\Gamma \cap \Delta)$ is either a trivial diagram or a singleton.
\end{definition}

In what follows, {\bf Haseman circles are not compressible}. We therefore only consider diagrams
that are neither trivial diagrams nor singletons.

\begin{definition}
A \textbf{twisted band diagram} (TBD) is a diagram homeomorphic to Fig. 3.
\end{definition}

The \textbf{signed weight} of a crossing on a band is defined according to
Fig. 2. It depends on the direction of the half-twist of the band supporting
the crossing.

\begin{figure}[ht]
\centering
\includegraphics[scale=0.3]{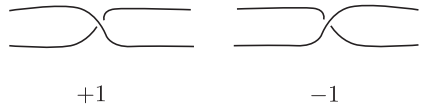}
\caption{The signed weight of a crossing on a band}
\end{figure}

\begin{figure}[ht]
\centering
\includegraphics[scale=.45]{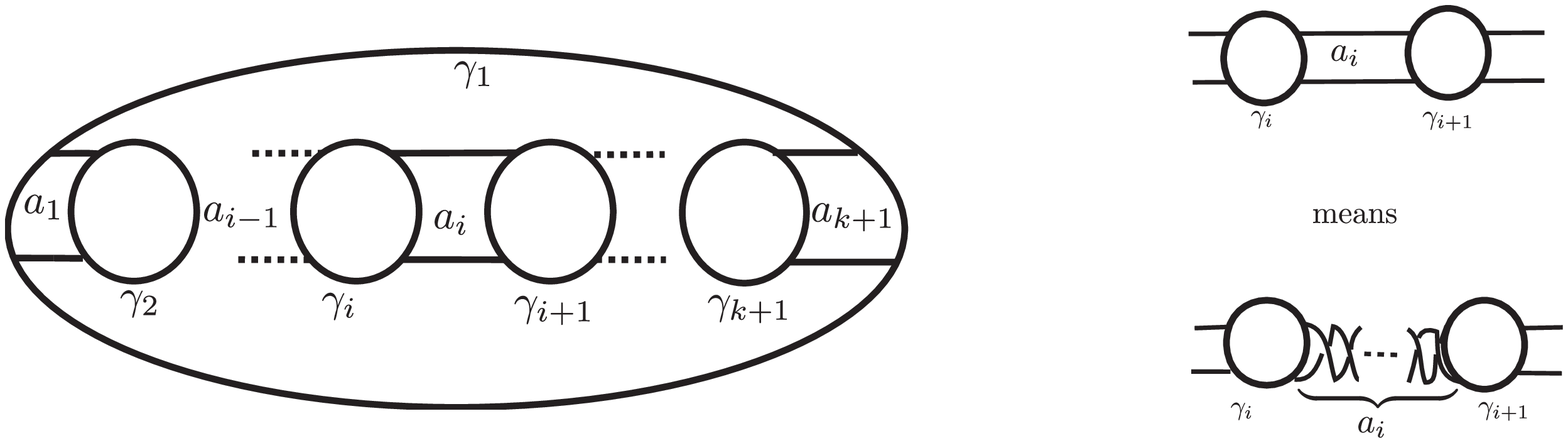}
\caption{ A twisted band diagram}
\end{figure}

In Fig. 3 the boundary components of $\Sigma$ are denoted $\gamma_1 , \dots
, \gamma_{k+1}$ where $k \geq 0$. The corresponding portion of the band
diagram between the projection and the circles $\gamma_i$ and $\gamma_{i+1}$
is called a \textbf{twist region} with $\vert a_i \vert$ crossing points.
The sign of $a_i$ is the signed weight of the $\vert a_i \vert$ crossing
points. The integer $a_i$ will be called an intermediate weight. 
\newline
If $k+1 = 1$, the planar surface $\Sigma$ is a disc and the twisted band
diagram $(\Sigma, \Sigma \cap \Pi)$ is called a \textbf{spire} with $\vert
a_1 \vert \geq 2$ crossings.\newline
If $k+1= 2$, the twisted band diagram is a \textbf{twisted annulus} and we
require that $a_1+a_2 \neq 0$.

We ask the crossings on the same band to have the same signed weight. In 
other words, using flypes (Fig. 7) and Reidemeister move 2, we can
reduce the number of crossing points of a twisted band diagram so that all
non zero intermediate weights $a_i$ of a twisted band diagram have the same
sign.

\begin{definition}
(1) The crossings of a TBD (twisted band diagram) $(\Sigma, \Sigma \cap \Pi)$
are called the \textbf{visible crossings} of $(\Sigma, \Sigma \cap \Pi)$.%
\newline
(2) The sum $a= \Sigma a_i$ is called the \textbf{total weight} of the
twisted band diagram $(\Sigma, \Sigma \cap \Pi)$. If $k+1\geq 3$ we may have 
$a = 0$. The absolute value of $a$ is equal to the number of the visible
crossings of $(\Sigma, \Sigma \cap \Pi)$.
\end{definition}
Two Haseman circles are said to be \textbf{parallel} if they bound an
annulus $A \subset \Sigma$ such that the pair $(A , A \cap \Gamma)$ is
diffeomorphic to Fig. 4.
\begin{figure}[ht]
\centering
\includegraphics[scale=0.2]{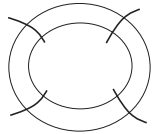}
\caption{Parallel Haseman circles}
\end{figure}

\noindent We define a Haseman circle $\gamma$ to be \textbf{boundary parallel%
} if there exists an annulus $A \subset \Sigma$ such that:\newline
(1) the boundary $\partial A$ of $A$ is the disjoint union of $\gamma$ and a
boundary component of $\Sigma$;\newline
(2) $(A , A \cap \Gamma)$ is diffeomorphic to Fig. 4.

\begin{definition}
A \textbf{jewel} is a diagram $J$ such that:\newline
(1) it is not a twisted band diagram with $k+1=2$ and $a=\pm1$ or with $%
k+1=3 $ and $a=0$.\newline
(2) each Haseman circle of $J$ is boundary parallel.
\end{definition}

Fig. 5 depicts a jewel $\mathrm{J= (\Sigma, \Pi \cap \Sigma)}$ where $\Sigma$
is a planar surface with boundary $\partial \Sigma= \gamma_1 \cup \gamma_2
\cup \gamma_3 \cup \gamma_4$.

\begin{figure}[ht]
\centering  
\includegraphics[scale=.3]{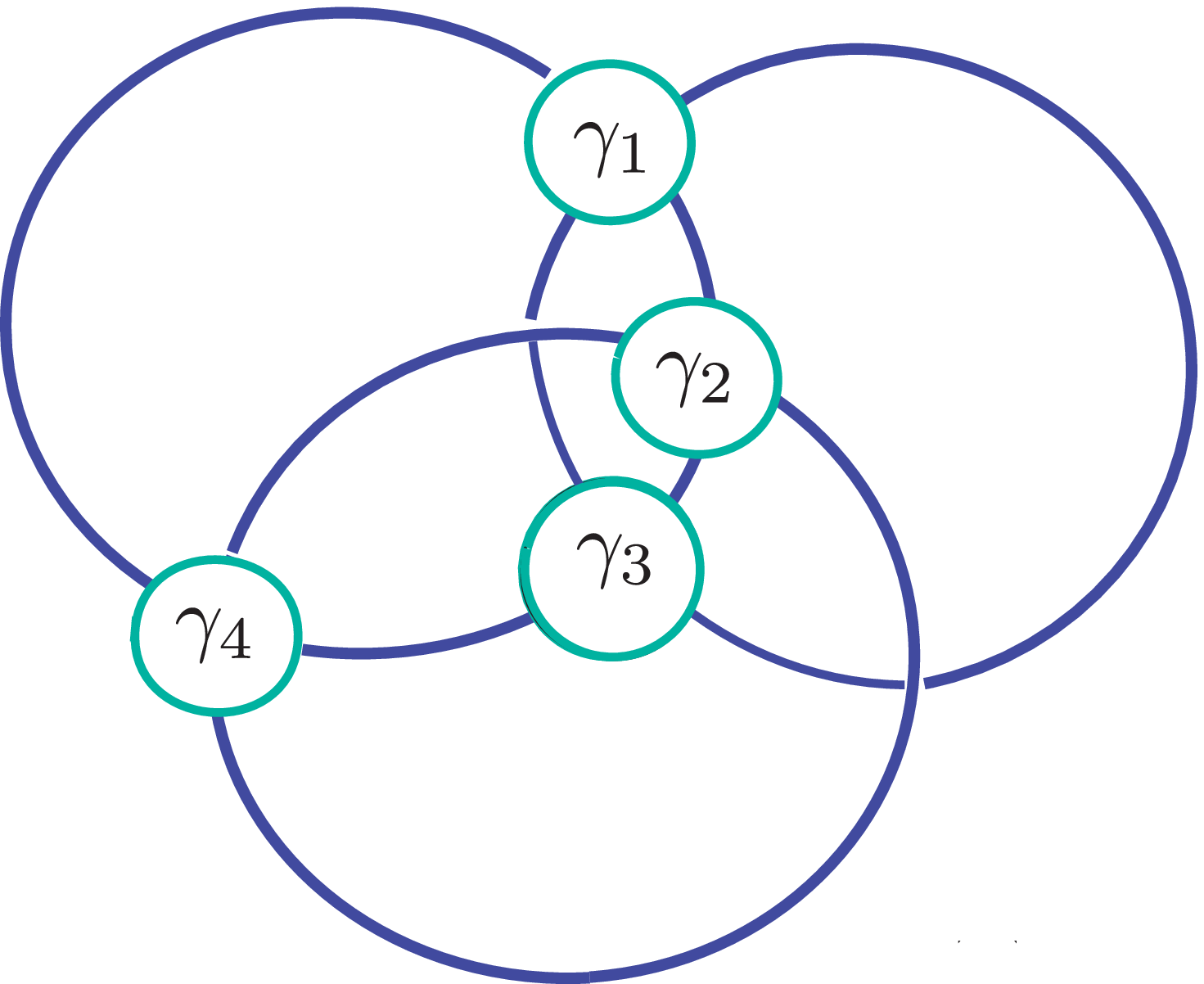}
\caption{A jewel}
\end{figure}

\subsection{Families of Haseman circles for a projection}

\subsubsection{Canonical Conway circles}

If not otherwise stated, the projections we consider are connected and prime.

\begin{definition}
Let $\Pi$ be a projection. A \textbf{family of Haseman circles} for $\Pi$ is
a set of Haseman circles satisfying the following conditions:\newline
(1) any two circles are disjoint and\newline
(2) no two circles are parallel.
\end{definition}

Let $\mathcal{H} = \{ \gamma_1 , ... , \gamma_n \}$ be a family of Haseman
circles for $\Pi$. Let $R$ be the closure of a connected component of $S^2
\setminus \bigcup_{i=1}^{i=n} \gamma_i$. We call the pair $(R , R \cap \Pi)$
a \textbf{diagram} of $\Pi$ determined by the family $\mathcal{H}$.

\begin{definition}
A family $\mathcal{C}$ of Haseman circles is an \textbf{admissible family}
if each diagram determined by it is either a twisted band diagram or a
jewel. An admissible family is \textbf{minimal} if removing a circle turns
it into a family that is not admissible.
\end{definition}

\noindent Theorem 2.1 is the main structure theorem about link projections
proved in (\cite{quwe0}, Theorem 1). It is essentially due to Bonahon and
Siebenmann.

\begin{figure}[h!]
\centering
\includegraphics[scale=.5]{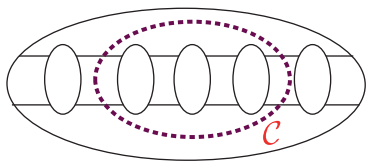}
\caption{A non canonical Conway circle }
\end{figure}
\vspace{.3cm}

\begin{theorem}
{(Existence and uniqueness theorem of minimal admissible families)} Let $\Pi$
be a link projection in $S^2$. Then: \newline
i) there exist minimal admissible families for $\Pi$; \newline
ii) any two minimal admissible families are isotopic by an isotopy
respecting $\Pi$.
\end{theorem}

\begin{definition}
An Haseman circle belonging to ``the" minimal admissible family of $\Pi$
noted $\mathcal{C}_{can}$ is called a \textbf{canonical Conway circle} of
the projection $\Pi$.

\begin{example}
The Haseman circle $C$ in Fig. 6 is not a canonical Conway circle.
\end{example}

The decomposition of $\Pi$ into twisted band diagrams and jewels determined
by $\mathcal{C}_{can}$ will be called the \textbf{canonical decomposition}
of $\Pi$. If there are no jewels in its canonical decomposition, the
projection $\Pi$ is said to be \textbf{arborescent}.
\end{definition}

\begin{figure}[h!]
\centering
\includegraphics[scale=.4]{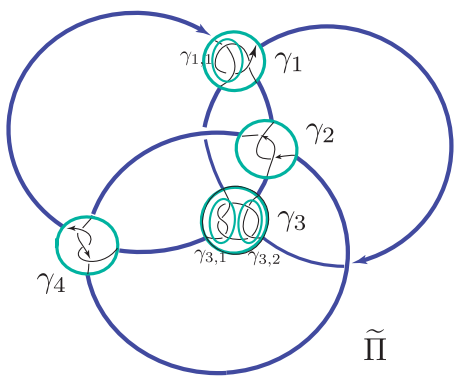}
\caption{A projection with its canonical Conway circles }
\end{figure}

A canonical Conway circle can be of 3 types:\newline
(1) a circle that separates two jewels.\newline
(2) a circle that separates two twisted band diagrams.\newline
(3) a circle that separates a jewel and a twisted band diagram.\newline

\begin{example}
Fig. 7 illustrates a projection $\widetilde \Pi$ with its canonical Conway
family: 
\begin{equation*}
\mathcal{C}_{can} (\widetilde \Pi)= \{ \gamma_1,\gamma_{1,1},\gamma_2,
\gamma_3,\gamma_{3,1},\gamma_{3,2}, \gamma_4 \}.
\end{equation*}
\end{example}

\begin{figure}[h!]
\centering
\includegraphics[scale=.4]{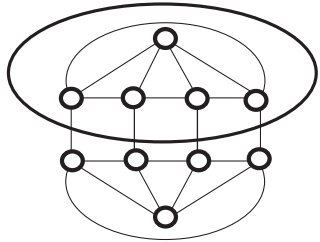}
\caption{$10^{***}$ is a tangle sum of two $6^*$}
\end{figure}
\vspace{.3cm}

\begin{remark}
\begin{enumerate}
\item As remarked in \cite{erquwe2}, our notion of jewel is more restrictive
than the notion of John Conway polyhedron (\cite{kaw} p. 139). We define a 
\textbf{jewel graph} $G_{J}$ of a jewel $J$ by collapsing each Haseman
circle of $J$ to a vertex. For John Conway, the graph of a basic polyhedron is
a simple regular graph of valency 4. A basic polyhedron can therefore be a
tangle sum of several jewel graphs. A jewel graph is simply a polyhedron in
the sense of John Conway, indecomposable with regard to the tangle sum. The
polyhedron $10^{\ast \ast \ast }$ has a non-trivial Haseman circle (see Fig.
8).

\item The minimal projection of the torus link of type $(2, m)$ can be
considered as a twisted band diagram with $k+1= 0$.
\end{enumerate}
\end{remark}

\subsubsection{Essential Conway circles}

Let $\Pi$ be a projection on $S^2$.

\begin{definition}
A (2-dimensional) tangle $\mathcal{T}$ of $\Pi $ is a pair $\mathcal{T}%
=(\Delta ,\tau _{\Delta })$ where $\Delta $ is a disk in $S^{2}$, $\tau
_{\Delta }$ is $\Pi \cap \Delta $ and the boundary $\partial \Delta $ of $%
\Delta $ intersects $\tau _{\Delta }$ exactly on 4 points. \newline
The \textbf{boundary} $\partial \mathcal{T}$ of $\mathcal{T}$ is the
boundary $\partial \Delta $ of $\Delta $.
\end{definition}

Note that a (2-dimensional) tangle is the projection onto the equatorial
disk of the 3-ball of a (3-dimensional) tangle which will be defined further
in Definition 3.4.

\begin{definition}
Two tangles $\mathcal{T}=(\Delta ,\tau _{\Delta })$ and $\mathcal{T}{%
^{\prime }}=(\Delta ,\tau _{\Delta }^{\prime })$ are \textbf{isotopic} if
there exists a homeomorphism $f:\mathcal{T}\rightarrow \mathcal{T}{^{\prime }%
}$ such that:\newline
(1) $f$ is the identity on the boundary $\partial \Delta $ and \newline
(2) $f(\tau _{\Delta })=\tau _{\Delta }^{\prime }$.
\end{definition}

\vspace{.2cm}

\begin{definition}
A \textbf{rational tangle} is a tangle such that all its \textbf{canonical
Conway circles} are concentric and delimit twisted annuli, with the
exception of the innermost circle which is the boundary of a spire, as shown
in Fig. 9.\newline
A \textbf{maximal rational tangle} of a link projection $\Pi$ is a rational
tangle that is not strictly included in a larger rational tangle of $\Pi$.
\end{definition}

Let $\mathcal{T}$ be a rational tangle. We now consider $\mathcal{T}$ under
the \textbf{cardan} form $\mathrm{T}[a_{0},\dots ,a_{m}]$ (or equivalently
under the standard form described in \cite{kala}) illustrated in Fig. 9 such
that the twisted band diagrams have weights $b_{i}=(-1)^{i}a_{i}$ with $%
i=0,\dots ,m$ and such that the first weight band $b_{0}$ is horizontal.

To $\mathrm{T}[a_0, \dots, a_m]$ where $a_0 \in \mathbb{Z}$ and $a_1, \dots
,a_m \in \mathbb{Z}-\{0 \}$, we assign the continued fraction 
\begin{equation*}
[a_0, a_2, \cdots a_m]=a_0+ \dfrac{1}{a_1 +\dfrac{1}{{\ddots} +\dfrac{1}{a_m}%
}}
\end{equation*}

If $\mathcal{T}$ is not the trivial tangle $\mathrm{T([\infty ])}$ (Fig. $%
1(a)$), the rational number ${\frac{r}{s}}=[a_{0},a_{2},\cdots a_{m}]$ with $%
(r,s)=1$ and $r>0$ is called the \textbf{fraction} $\mathrm{F(}\mathcal{T}%
\mathrm{)}$.\newline
By convention, the fraction of the trivial tangle $\mathrm{T([\infty ])}$ is: 
$\mathrm{F(T([\infty ]):=\infty }$.

The fraction is an isotopy invariant of the tangle $\mathcal{T}$. It means
that with the expansion of $\frac{r}{s}$ in another continuous fraction $%
[d_{0},\dots ,d_{k}]$, we get another cardan tangle $\mathrm{T}%
[d_{0},\dots ,d_{k}]$ isotopic to $\mathrm{T}[a_{0},\dots ,a_{m}]$.
We will use $\mathrm{T}_{\frac{r }{s}}$ to denote a rational tangle with fraction $\frac{r}{s}$.

\begin{figure}[ht]
\centering
\includegraphics[scale=0.4]{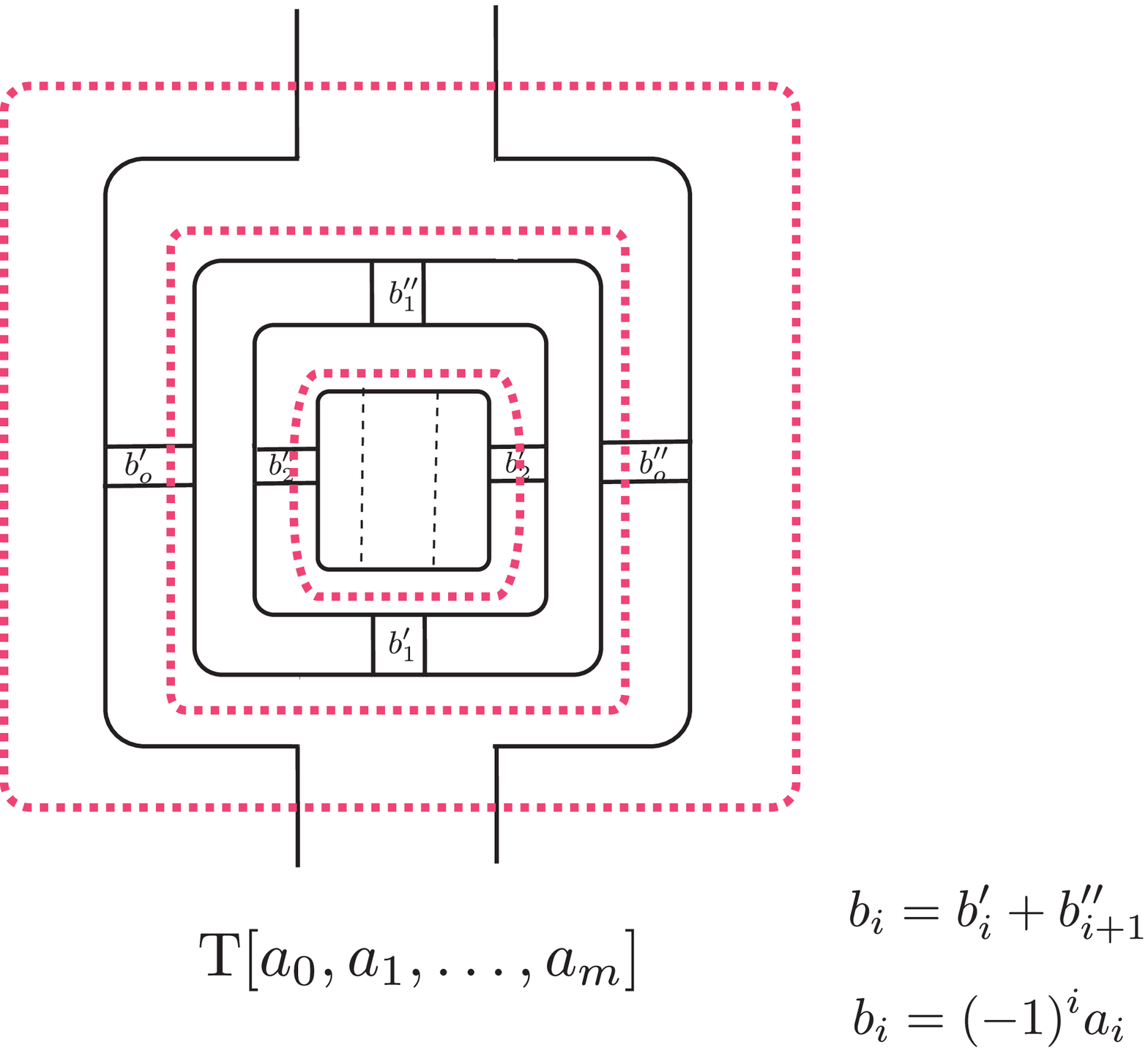}
\caption{$\mathrm{T}[a_0, \dots, a_m]$}
\end{figure}
\vspace{.3cm}

\begin{remark}
Let ${\frac{r }{s}}$ be a rational number with $r>0$ and $(r,s)=1$. Then ${%
\frac{r }{s}}$ has an expansion $[a_0, \dots,a_m]$ where the $a_i$'s are all positive or all
negative; it is called a {\bf homogeneous continued fraction}.
 If furthermore, $a_0$ and $a_m$ are not equal to $\pm 1$, the
continued fraction $[a_0, \dots,a_m]$ is said to be \textbf{strictly homogeneous}%
. If $[a_0, \dots, a_m]$ is a homogeneous continued fraction, the
cardan tangle $\mathrm{T}[a_0, \dots, a_m]$ is an alternating
tangle.
\end{remark}

\begin{definition}
An \textbf{essential Conway circle} of an alternating projection $\Pi $ is a
canonical Conway circle that is not properly contained in a maximal rational
tangle.
\end{definition}

In a rational link projection, there are no essential Conway circles.

Let $\Pi $ be a non-rational link projection. By removing from the minimal
admissible family $\mathcal{C}_{can}$ of $\Pi $ all concentric Conway
circles of each maximal rational tangle $\mathcal{T}$ of $(S^{2},\Pi )$
except its boundary circle $\partial \mathcal{T}$, we obtain the \textbf{%
essential Conway family} of $\Pi $ denoted $\mathcal{C}_{ess}(\Pi )$.

\begin{remark}
The set of essential Conway circles is empty for an alternating projection
if only if the projection is one of the three following cases:\newline
a) a standard torus knot projection of type $(2, s)$ (in this case it can be considered
as a twisted band diagram with empty boundary (Remark 2.2.2)), \newline
b) a jewel without boundary \newline
c) a minimal projection of a rational knot.
\end{remark}

\begin{figure}[ht]
\centering
\includegraphics[scale=.35]{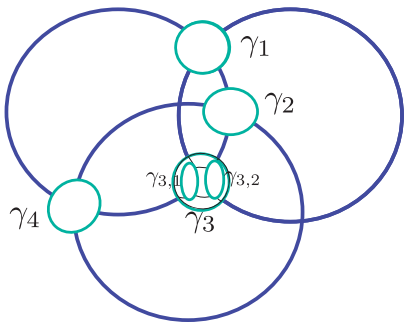}
\caption{The essential Conway family of the projection $\protect\widetilde %
\Pi$ illustrated in Fig. 7}
\end{figure}

\begin{figure}[h!]
\centering
\includegraphics[scale=.4]{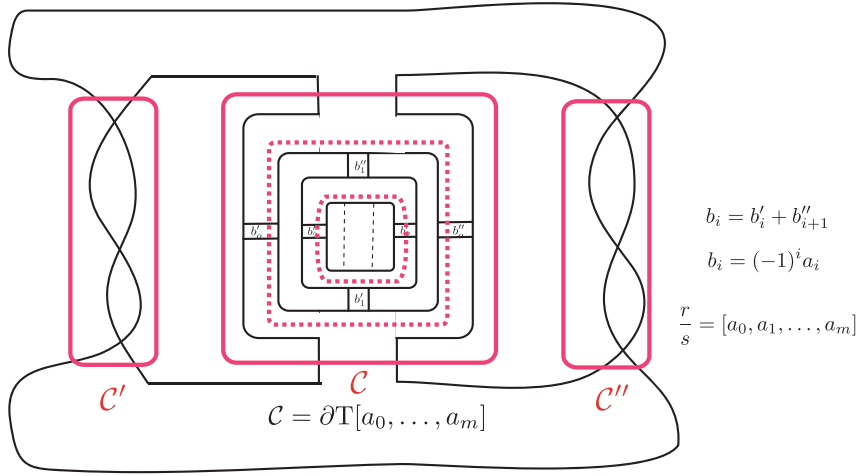}
\caption{A projection with its set $\mathcal{C}_{ess}=\{\mathcal{C}$, $%
\mathcal{C^{\prime }}$, $\mathcal{C^{\prime \prime }\}}$}
\end{figure}

\begin{example}
1) Fig. 10 illustrates the essential Conway family $\mathcal{C}%
_{ess}(\widetilde \Pi)$ of the projection $\widetilde \Pi$ of Fig. 7: 
\begin{equation*}
\mathcal{C}_{ess}(\widetilde \Pi) =\{ \gamma_1,\gamma_2,
\gamma_3,\gamma_{3,1},\gamma_{3,2},\gamma_4 \}.
\end{equation*}

2) In the projection depicted in Fig. 11, $\mathcal{C^{\prime }}$, $\mathcal{%
C^{\prime \prime }}$ and $\mathcal{C}=\partial\mathrm{T}[a_0, \dots, a_m]$
are essential Conway circles while dotted circles are just canonical Conway
circles.
\end{example}

\subsection{Canonical and Essential Structure Trees}

We now focus the canonical decomposition of alternating link projections.

\subsubsection{Flypes and Flyping Theorem}

Let $\Pi $ be a $n$-crossing regular projection of a link $L\hookrightarrow
S^{3}$ on the projection plane $S^{2}$. As in \cite{meth}, consider $n$
disjoint small \textquotedblleft crossing ball" neighbourhoods $B_{1},\dots
,B_{n}$ of the crossing points $c_{1},\dots c_{n}$ of $\Pi $. Then assume
that $L$ coincides with $\Pi $, except that inside each $B_{i}$ the two arcs
forming $\alpha (c_{i})=\Pi \cap B_{i}$ are perturbed vertically to form
semicircular overcrossing and undercrossing arcs, which lie on the boundary
of $B_{i}$. This relationship between the link $L$ and its projection $\Pi $
is expressed as $L=\lambda (\Pi )$ (\cite{meth}). Note that there is a
homeomorphism of pairs $(S^{3},L)\rightarrow (S^{3},\lambda (\Pi ))$. We call $\lambda (\Pi )$ a \textbf{realized projection} (or a realized diagram in the terms of \cite{boy}) for $L$.

We can consider the ambient space $S^{3}$ to be ${\mathbb{R}^{3}}\cup
\{\infty \}$, and we shall take the 2-sphere $S^{2}$ on which the regular
projection $\Pi $ lies to be $S^{2}=\{\mathbf{x}\in \mathbb{R}%
^{3}:\left\Vert \mathbf{x}\right\Vert =1\}$. We assume that the knot $%
\lambda (\Pi )$ lies within the neighborhood $N=\{x\in \mathbb{R}^{3}:{\frac{%
1}{2}}\leq \left\Vert {\mathbf{x}}\right\Vert \leq {\frac{3}{2}}\}$.

\begin{definition}
Let $g:(S^{3},\lambda (\Pi _{1}))\rightarrow (S^{3},\lambda (\Pi _{2}))$ be
a homeomorphism of pairs. The homeomorphism of pairs $g$ is \textbf{flat} if 
$g$ is pairwise isotopic to a homeomorphism of pairs $h$ with the condition
that $h$ maps $N$ onto itself and $h|_{N}=h_{0}\times id_{[{\frac{1}{2}},{%
\frac{3}{2}}]}$ for some orientation-preserving homeomorphism $%
h_{0}:S^{2}\rightarrow S^{2}$. We call $h_{0}$ the \textbf{principal part
of the flat homeomorphism} $g$.
\end{definition}

By asking the crossing balls to be sent on the crossing balls, a flat
homeomorphism of pairs is an isomorphism of realized projections as defined in 
\cite{boy}.

\begin{definition}
An \textbf{isomorphism of realized projections} $h:\lambda (\Pi )\rightarrow
\lambda (\Pi _{\ast })$ is a homeomorphism of pairs $h:(S^{3},L)\rightarrow (%
\tilde{S}^{3},\widetilde{L})$ such that:\newline
(1) $h(S^{2})={S}^{2}$,\newline
(2) $h(B_{i})=\widetilde{B}_{i}$ and \newline
(3) $h(\alpha (c_{i}))=\alpha (\tilde{c}_{i})$.
\end{definition}

We recall the definition of a flype as described in \cite{meth}.

\begin{definition}
Let $\Pi _{1}$ be a projection with the pattern described in Fig. 12(a). A 
\textbf{standard flype} of $(S^{3},\lambda (\Pi _{1}))$ is
any homeomorphism $f$ which maps $(S^{3},\lambda (\Pi _{1}))$ to a pair $%
(S^{3},\lambda (\Pi _{2}))$ where $\Pi _{2}$ is the pattern described in
Fig. 12(b), in such a way that:\newline
(1) $f$ sends the 3-ball $B_{\mathrm{A}}$ into itself by a rigid rotation
about an axis in the projection plane,\newline
(2) $f$ fixes pointwise the 3-ball $B_{\mathrm{B}}$,\newline
(3) $f$ moves the crossing visible on the left of Fig. 12(a) to the crossing
visible on the right of Fig. 12(b).
\end{definition}

\begin{definition}
Let $\Pi _{1}$ be any projection. Then a flype is any homeomorphism $%
f:(S^{3},\lambda (\Pi _{1}))\rightarrow (S^{3},\lambda (\Pi _{1}))$ of the
form $f=g_{1}\circ f^{\prime }\circ g_{2}$ where $f^{\prime }$ is a standard
flype and $g_{1}$ and $g_{2}$ are flat homeomorphisms.
\end{definition}

If the tangle $A$ of Fig. 12 contains no crossing, then the standard flype determined by that figure is a flat homeomorphism; therefore
according to the above definition, any flat homeomorphism is a flype. We call it a \textbf{trivial flype}.

A 2-dimensional description of a standard flype, as in Fig. 12, is
sufficient since it corresponds to a unique standard flype up to isomorphism
of realized diagrams. The crossing that moves during the (2-dimensional)
flype is an active crossing point of the flype.

\textbf{Terminology}. If there is no possible confusion, we will also
designate the (2-dimensional) flype by flype.

\begin{figure}[h!]
\centering
\includegraphics[scale=.3]{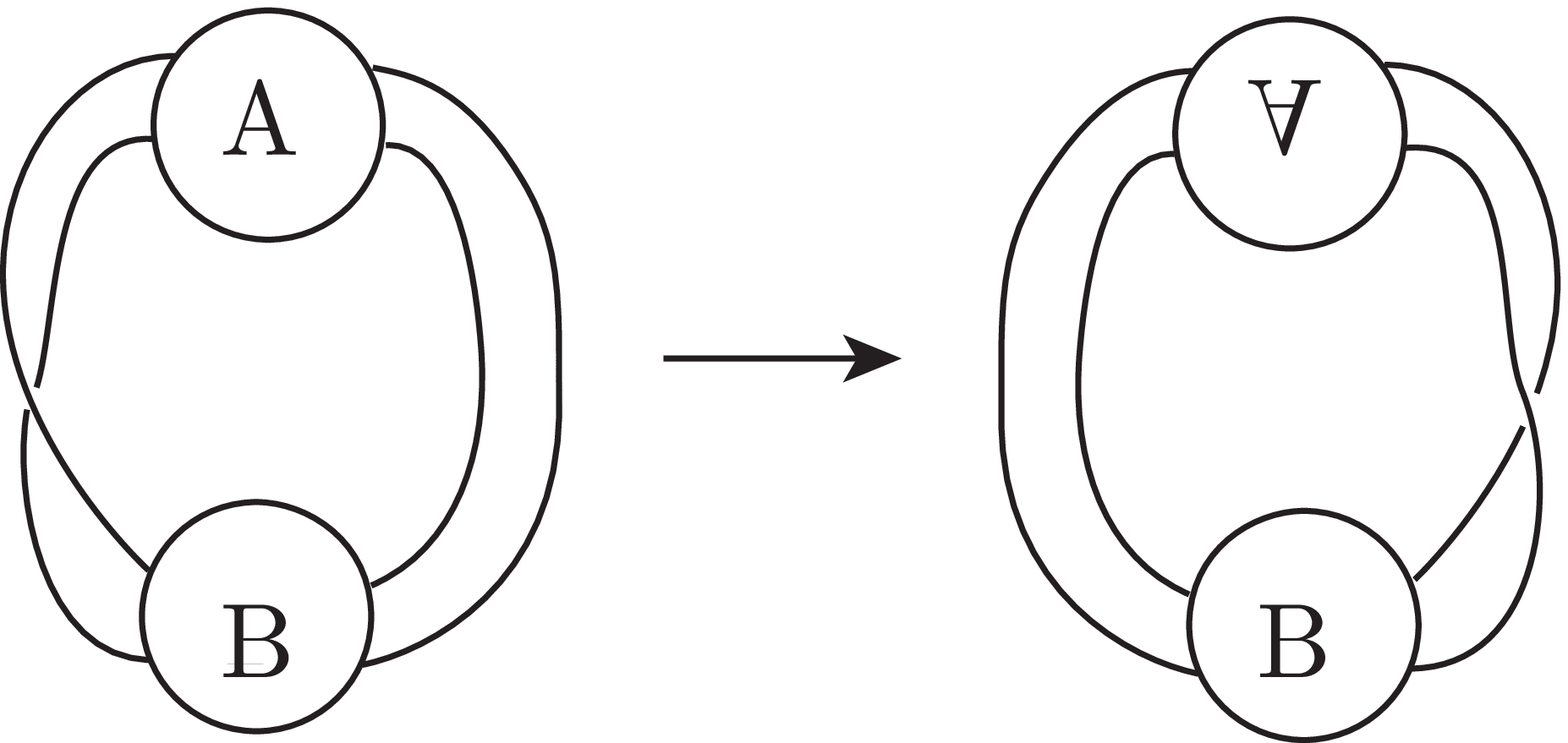}
\caption{A flype}
\end{figure}

We can now precisely locate where flypes can be performed with respect to
the canonical Conway decomposition of a prime alternating reduced link.

\begin{theorem}
\cite{quwe0} (Position of flypes) Let $\Pi$ be a prime alternating reduced
link projection in $S^2$ and suppose that a flype can be done in $\Pi$.
Then, its active crossing point belongs to a diagram determined by $\mathcal{%
C}_{can}(\Pi)$. The flype moves the active crossing point either within the twist
region to which it belongs or to another twist region of the same twisted
band diagram.
\end{theorem}

\begin{figure}[h!]
\centering
\includegraphics[scale=.4]{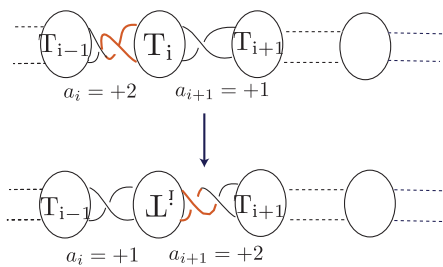}
\caption{ An efficient flype }
\end{figure}

\begin{remark}
(1) We are only interested in \textbf{efficient} flypes that move the active
crossing point from one twist region to another in the same twist band
diagram.

(2) Let $\mathcal{T}$ be a TBD of $\Pi$ and $C_1, C_2, \dots, C_k, C_1$ be
its cyclic sequence of $\mathcal{C}_{can}|_{\mathcal{T}}$. A flype on $%
\mathcal{T}$ does not modify the order of occurence of the Conway circles $%
\mathcal{C}_{can}|_{\mathcal{T}}$ in the cyclic sequence.
\end{remark}

\begin{definition}
(1) The set of the twist regions of a given twisted band diagram is called a 
\textbf{flype orbit} (Fig. 14).\newline
(2) The cardinal of $\mathcal{C}_{can}|_{\mathcal{T}}$ is the \textbf{valency%
} $k+1$ of the TBD $\mathcal{T}$.
\end{definition}

\begin{figure}[h!]
\centering
\includegraphics[scale=.3]{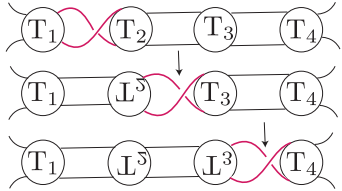}
\caption{ A flype orbit }
\end{figure}

\begin{corollary}
\cite{quwe0}\newline
(1) A flype moves an active crossing point inside the flype orbit to which
it belongs.\newline
(2) Two distinct flype orbits are disjoint.
\end{corollary}

This implies that an active crossing point belongs to one and only one TBD.
Since two TBD have at most one canonical Conway circle, Corollary 2.1 can be
interpreted as a loose kind of commutativity of flypes.

\subsubsection{Canonical Structure Tree ${\mathcal{A}}(L)$.}

Fundamental to our purposes is the following Menasco-Thistlethwaite Flyping
Theorem \cite{meth}:

Let $\Pi _{1}$ and $\Pi _{2}$ be two reduced, prime, oriented, alternating
projections of links. If $f:(S^{3},\lambda (\Pi _{1}))\rightarrow
(S^{3},\lambda (\Pi _{2}))$ is a homeomorphism of pairs, then $f$ is a
composition of flypes and flat homeomorphisms.

Since two realizations of alternating projections in $S^{2}$ of the same
isotopy class of an oriented prime alternating link in $S^{3}$ are related by
flypes, their canonical and essential \textbf{structure trees} constructed
as described below, are isomorphic.

\textit{Construction of the canonical structure tree ${\mathcal{A}}(L)$}.

Let $L$ be a prime alternating link and let $\Pi $ be an alternating
projection of $L$. Let $\mathcal{C}_{can}$ be the canonical Conway family
for $\Pi $. We construct the canonical structure tree ${\mathcal{A}}(L)$ as
follows: its vertices are in bijection with the diagrams determined by $%
\mathcal{C}_{can}$ and its edges are in bijection with the canonical circles;
 the vertices of each edge represent the diagrams which both have in their boundary, the canonical
circle $\gamma $ corresponding to the edge. Since $S^{2}$ has genus zero,
the constructed graph is a tree. \newline
We label the vertices of ${\mathcal{A}}(L)$ as follows: if a vertex
represents a twisted band diagram, we label it by its total weight $a$ and
if it represents a jewel, we label it with the letter $J$.

In the case of a tangle $\mathcal{T}$ whose boundary is a canonical Conway
circle $\gamma $, the canonical structure tree ${\mathcal{A}}(\mathcal{T})$
of $\mathcal{T}$ is a graph such that all its edges have two vertices at the
extremities except for an\textquotedblleft open" edge (with a single vertex)
which represents the circle $\gamma $. For an example, see Fig. 16.

\begin{proposition}
The canonical structure tree ${\mathcal{A}}(L)$ is independent of the
alternating projection chosen to represent $L$.
\end{proposition}

\begin{proof} Let $\Pi$ be an alternating link projection in $S^2$. By \cite{quwe0}, Theorem 1:\\
i) there exist minimal Conway families for $\Pi$ and \\
ii) any two minimal Conway families are isotopic, by an isotopy which respects $\Pi$.

A flype or a flat homeomorphism does not modify the canonical structure tree. 
By Flyping Theorem, we conclude that the canonical structure tree is independent of the chosen alternating projection $\Pi$ and we can speak of the canonical structure tree of $L$ (and not only of $\Pi$).
\end{proof}

\begin{definition}
The alternating knot $K$ is \textbf{arborescent} if each vertex of ${%
\mathcal{A}} (K)$ has an integer weight.
\end{definition}

\begin{example}
The link $K_0$ which has the projection $\widetilde \Pi$ represented by Fig.
7, has its canonical structure tree ${\mathcal{A}} ( K_0)$ given by Fig. $%
15(a)$.
\end{example}

\begin{figure}[tbp]
\centering
\includegraphics[scale=.26]{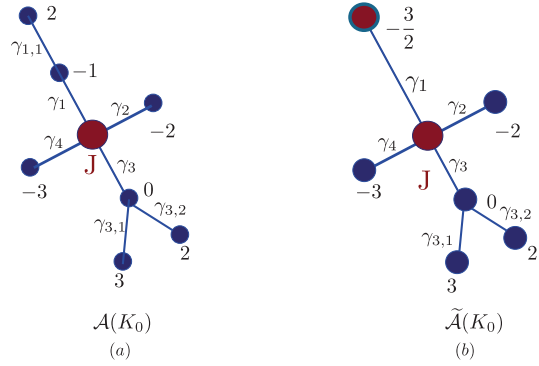}
\caption{Canonical structure tree and essential structure tree of $K_0$}
\end{figure}

\begin{remark}
If the projection $\Pi$ is arborescent, we can encode $\Pi$ with a \textbf{%
weighted planar tree} {\`a} la Bonahon-Siebenman (\S 5 in \cite{quwe0})
which is a canonical structure tree with more complete information.
\end{remark}


\subsubsection{Essential Structure Tree $\mathcal{\tilde{A}}(L)$.}

\textit{Construction of the essential structure tree $\mathcal{\tilde{A}}(L)$}.

On the same lines of the construction of the canonical structure tree ${%
\mathcal{A}}(L)$, we construct the \textbf{essential structure tree} $%
\widetilde{\mathcal{A}}(L)$. The vertices of $\widetilde{\mathcal{A}}$ are
in bijection with the diagrams determined by the set $\mathcal{C}_{ess}(\Pi
) $ and the edges are in bijection with the circles of $\mathcal{C}_{ess}(\Pi )$%
. The extremities of an edge corresponding to an essential Conway circle $%
\gamma $ are two vertices associated to the two diagrams having $\gamma $ in
their boundary.

As in the case with the canonical structure tree, Flyping Theorem implies
that:

\begin{proposition}
The essential structure tree ${\mathcal{\widetilde{A}}}(L)$ is independent
of the minimal projection chosen to represent $K$.
\end{proposition}

\textbf{The essential structure tree of a tangle}: To a tangle $\mathcal{T}$
with an essential Conway circle $\gamma $ as boundary, we associate an
essential structure tree denoted $\widetilde{\mathcal{A}}(\mathcal{T})$
which has only one \textquotedblleft open edge" with one vertex-end. The
unique edge of $\widetilde{\mathcal{A}}(\mathcal{T})$
corresponds to $\gamma $.

\begin{remark}
\begin{enumerate}
\item If $\mathrm{T}_{\frac{p }{q}}$ is a maximal rational tangle, $\widetilde {%
\mathcal{A}} ({\mathrm{T}_{\frac{p }{q}}})$ is a linear graph composed only with an
``open" edge and one vertex labelled by ${\frac{p }{q}}$.

\item A vertex in ${\mathcal{\widetilde{A}}}(K)$ with weight $\in \mathbb{Q}%
\setminus \mathbb{Z}$ is monovalent and its union with its single edge
corresponds to a maximal rational tangle in $\Pi $.

\item Only monovalent vertices of an essential structure tree of a link can
have weights that are $\in \mathbb{Q}\setminus \mathbb{Z}$
\end{enumerate}
\end{remark}

The essential structure tree ${\mathcal{\widetilde A}} (\mathrm{T}_{\frac{r }{s}})$
is reduced to a vertex of weight $\frac{r }{s}$.

\begin{figure}[h!]
\centering  
\includegraphics[scale=0.3]{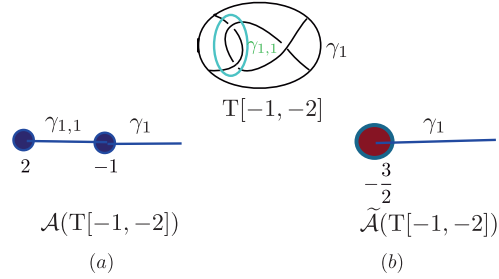}
\caption{$(a)$ Canonical structure tree of $\mathcal{A}(\mathrm{T}[-1,-2])$
and $(b)$ Essential structure tree of $\protect\widetilde {\mathcal{A}} (%
\mathrm{T}[-1,-2])$}
\end{figure}

\begin{example}
$\mathcal{A}(\mathrm{T}[-1,-2])$ and $\widetilde{\mathcal{A}} (\mathrm{T}%
[-1,-2])$ are described in Fig. 16.
\end{example}

\begin{remark}
The essential structure tree ${\mathcal{\widetilde A}} (K)$ is reduced to a
single vertex if and only if $K$ is a rational link $\mathrm{T}_{\frac{r }{s}}$ or is a link described by a jewel without boundary.
\end{remark}

\section{On Visibility Theorem 3.1}

This section is about the proof of the Visibility Theorem 3.1 for $q$%
-periodic alternating prime knots:

\begin{theorem}
Let $K$ be an oriented prime alternating knot that is $q$-periodic with $%
q\geq 3$. Then there exists a $q$-periodic alternating projection $%
\widetilde{\Pi }$ for $K$.\newline
\end{theorem}

We first recall the definition of a $q$-periodic knot in $S^3$.

\begin{definition}
A knot $K$ is $q$-periodic if there is a (auto)-homeomorphism $\Phi $ of
pairs $(S^{3},K)$ of period $q$ which satisfies the following conditions:%
\newline
(1) $\Phi $ is a $\frac{2\pi }{q}$-rotation about a \textquotedblleft line"
(circle) $\alpha $ in $S^{3}$ and \newline
(2) $\alpha \cap K=\varnothing $.\newline
$\Phi $ is called a $q$-homeomorphism of $(S^{3},K)$.
\end{definition}

\begin{remark}
Let $K$ be a $q$-periodic knot with $\Phi$ a $q$-homeomorphism of $(S^{3},K)$%
. For each divisor $p$ of $q$, $K$ is $p$-periodic with its $p$-homeomorphism $\Psi _{r}=\Phi ^{r}$ where $r={\frac{q}{p}}$.
\end{remark}

For our ends, we introduce the following notion:

\begin{definition}
 1) If a knot $K$ is $q$-periodic but not \textbf{strictly $q$-periodic} then $K$
is $r q$-periodic for some $r \geq 2$.\\
 2) If a projection $\Pi$ is $q$-periodic but not strictly $q$-periodic then $\Pi$ 
is $r q$-periodic for some $r \geq 2$.
\end{definition}

\begin{figure}[tbp]
\centering
\includegraphics[scale=.3]{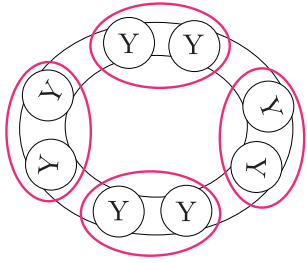}
\caption{A non-strictly 4-periodic projection}
\end{figure}

\begin{example}
Let $\Pi $ be an alternating projection described in Fig.17 where $\mathrm{Y}
$ is an alternating tangle with boundary an essential Conway circle. The big
red circles do not belong to $\mathcal{C}_{ess}(\Pi )$ and hence do not
appear on $\mathcal{\ }\tilde{A}(K)$. The projection $\Pi $ is
 not strictly 4-periodic since as shown by Fig. 17, $\Pi $
is an 8-periodic projection.\newline
Note that the projection $\Pi $ depicted in Fig.17 is not a knot projection,
regardless the tangle $\mathrm{Y}$ (see Proposition 3.2 below).
\end{example}

From now on, by \newline

\centerline {{\bf $q$-periodic projections}, we mean {\bf strictly
$q$-periodic projections}.}

Our objective being the study of periodicity, we reformulate the Flyping
theorem in the following form:

\begin{theorem}
 Let $\Phi:(S^{3},K)\rightarrow (S^{3},K)$ be an orientation preserving homeomorphism
of pairs where $K$ is a prime alternating knot. Let $\lambda (\Pi )$ be a
realized projection of a reduced alternating projection $\Pi$ of $K$ and $h:(S^{3},K)\rightarrow (S^{3},\lambda (\Pi ))$ be
a homeomorphism of pairs. Then the isomorphism of realized projections $\Phi _{\Pi }=h$ $\circ $ 
$\Phi $ $\circ $ $h^{-1}:(S^{3},\lambda (\Pi ))\rightarrow
(S^{3},\lambda (\Pi ))$ can be expressed as $\Phi _{\Pi }=\phi \circ F$ where
$\phi$ is a flat homeomorphism and $F$ is a composition of
standard flypes on $\lambda (\Pi )$ unless $F$ is the identity.
\end{theorem}

\begin{remark}
\begin{enumerate}
\item For our purposes, we separate the flat homeomorphisms from the flypes. Therefore, the
standard flypes involved in Flyping Theorem 3.2 above are not trivial. 
\newline
\item Standard flypes
and flat homeomorphisms \textquotedblleft essentially" commute. 
By \textquotedblleft essentially" commute, we mean that if $f$ is a
standard flype and $h$ is a flat homeomorphism, then $f\circ h=h\circ
f^{\prime }$ where $f^{\prime }=h^{-1}\circ f\circ h$ is also a standard
flype. Then it is possible to express $\Phi _{\Pi }=\phi \circ F$.

\end{enumerate}
\end{remark}

Let $K\subset S^{3}$ be a prime (strictly) $q$-periodic alternating knot
with $\Phi :(S^{3},K)\rightarrow (S^{3},K)$ its corresponding rotation about an axis of
order $q$. Let $\Pi $ be a reduced alternating projection of $K$ and $\lambda (\Pi )$ its
realized diagram. By Flyping Theorem, $\Phi $ is conjugate
through maps of pairs to an isomorphism $\Phi _{\Pi }$ on  
$\lambda (\Pi )$ (onto itself) which is a composition of a
 flat homeomorphism with standard flypes.
  
\begin{definition}
An \textbf{essential Conway sphere} of $\Pi $ is a 2-dimensional sphere
lying in the interior of $N=\{x\in \mathbb{R}^{3}:{\frac{1}{2}}\leq
\left\Vert {\mathbf{x}}\right\Vert \leq {\frac{3}{2}}\}$ such that its
projection on the projection plane is an essential Conway circle.\newline
\end{definition}
Consider the set $\mathcal{C}_{ess}(\Pi )$ of essential Conway circles of $%
\Pi $ and its corresponding set $\mathcal{S}_{ess}(\Pi )$ of essential
Conway spheres.

A  2-dimensional tangle in Definition 2.10 is the projection
on the equatorial disk of the 3-ball $B$ of a 3-dimensional tangle defined
as follows:

\begin{definition}
A 3-dimensional tangle is a pair $(B,t)$, where $B$ is a 3-ball and t is a
proper 1-submanifold of $B$ meeting $\partial B$ in four points. We say that
two tangles are equivalent if they are homeomorphic as pairs.\\ We say that $%
(B,t)$ is trivial if it is equivalent to the pair $(B_{0},t_{0})$, where $%
B_{0}=\{{(x_{1},x_{2},x_{3})\in \mathbb{R}^{3}:x_{1}^{2}+x_{2}^{2}+x_{3}^{2}%
\leq 1\}}$ and $t_{0}$ consists of the points of $B_{0}$ for which $x_{1}\in
\{{1/2,-1/2}\}$ and $x_{3}=0$.
\end{definition}

From such a tangle diagram $T$, we can create a 3-dimensional tangle $\lambda
(T)$ by means of a suitably small vertical perturbation near each crossing
of the diagram; the ambient space of $\lambda (T)$ is considered as a 3-ball
for which the disk region of $T$ is an equatorial slice.

\begin{corollary}
$\Phi _{\Pi }$ induces a permutation $\sigma _{\Phi _{\Pi }}$ on the set of
essential Conway spheres $\mathcal{S}_{ess}(\Pi )$ such that $\sigma _{\Phi
_{\Pi }}^{q}$ is the identity permutation.
\end{corollary}

\begin {proof} 
The corollary is deduced from the following facts:\\
1) $\Phi _{\Pi }$ sends crossing balls to crossing balls and therefore induces a non trivial permutation $\sigma (\Phi _{\Pi })$ on the set of crossing balls of $\lambda (\Pi )$ and likewise a non trivial permutation of the set of crossings of $\Pi $  such that $\sigma (\Phi _{\Pi })^q$ is the identity permutation.\\
2) the existence and unicity of the family of essential Conway circles ${\cal C}_{ess}(\Pi)$ of $\Pi$ (Theorem 2.1) imply the existence and unicity (up isotopy) of the family of essential Conway spheres. 
\end {proof}

Denote by $\tilde{\varPhi}$ the automorphism induced by $\Phi_\Pi$ on the
tree $\tilde{ \mathcal{A}} (K)$.

\begin{corollary}
The automorphism $\tilde{\varPhi}$ on $\widetilde{\mathcal{A}}(K)$ satisfies 
$\tilde{\varPhi}^q=\mathrm{Identity}$.
\end{corollary}

\begin{proof} 
By Corollary 3.1, the permutation $\sigma(\Phi_\Pi)^q$ on ${\cal C}_{ess}(\Pi)$ is the identity permutation on ${\cal C}_{ess}(\Pi)$. Hence it induces the identity on the essential structure tree $\widetilde{\mathcal{A}}(K)$.
\end{proof}

\subsection{Visibility of the $q$-periodicity of alternating knots on $S^2$.}

We now define the notion of visibility of a $q$-periodicity of an
alternating knot.

\begin{definition}
Let $K$ be an alternating $q$-periodic knot. The $q$-periodicity of $K$ is 
\textbf{visible} if $K$ displays the $q$-periodicity of $K$ as a $\frac{2\Pi 
}{q}$-rotation on an \textbf{alternating projection} called a $q$-visible
projection.
\end{definition}

In \S 3.2 and \S 3.3, we will describe how the $q$-periodicity of an
alternating knot $K$ acts on the structure trees by
studying how it is reflected on the set of essential Conway circles as
well as on the diagrams of $\Pi (K)$.

According to Flyping Theorem, we have two cases:

(1) Suppose no flypes are needed. Hence $\Phi _{\Pi }=\phi $ is flat and $\phi ^{q}=%
\mathrm{Id}$. By Ker\'{e}kj\'{a}rto's theorem (\cite{coko}), the principal part of $%
\phi $ which is a homeomorphism of $(S^{2},\Pi )$ is topologically conjugate to a
rotation of $S^{2}$ of order $q$ without fixed points on $\Pi $. Consequently,
the $q$-periodicity of $K$ is visible on an alternating projection $\Pi $ of 
$K$.

\begin{remark}
\begin{enumerate}
\item If $(S^{2},\Pi )$ is a $q$-periodic jewel without boundary, its $q$%
-periodicity is visible on $\Pi $ because the jewels have no TBD and then no
flypes are necessary to realize $\Phi _{\Pi }$.

\item A torus knot of type $(2,q)$ ($q$ must therefore be odd) displays the $q$%
-periodicity on a standard alternating projection.
\end{enumerate}
\end{remark}

\vspace{.2cm}

(2) In what follows, we will deal with the case where flypes may be involved.

\subsubsection{Action of $\Phi _{\Pi }$ on Structure Trees.}

Let $\Pi $ be a reduced alternating projection of $K$ and $\mathcal{C}_{ess}(\Pi )$ 
be its set of essential Conway circles. Suppose further that $\mathcal{C}_{ess}(\Pi )$ is not empty.

Since the essential structure trees $\widetilde{\mathcal{A}}(K)$ and $\widetilde{\mathcal{A}}(\Phi _{\Pi }(K))$ of a prime
oriented alternating $q$-periodic knot $K$ are isomorphic graphs, we can
interpret this isomorphism as an automorphism $\widetilde{\varPhi}$ of the
essential structure tree $\widetilde{\mathcal{A}}(K)$.

Since the graph $\widetilde{\mathcal{A}}(K)$ is a tree, the fixed point set $%
Fix (\widetilde \varPhi) $ is a non-empty subtree. So we have two
possibilities:

\begin{enumerate}
\item Case where $Fix (\widetilde \varPhi)$ contains an edge $E$, corresponding to a Conway circle which is $\Phi _{\Pi }$ invariant.

\item Case where $Fix (\widetilde \varPhi)$ is reduced to a vertex $V_0$.
\end{enumerate}

\begin{remark}
If $\widetilde{\mathcal{A}}(K)$ is reduced to a single vertex $V_{0}$ then $%
K$ is either a rational link or a link corresponding to a jewel without
boundary and the automorphism $\widetilde{\varPhi}$ is obviously the identity
map.
\end{remark}

\subsection{$Fix (\protect\widetilde \varPhi)$ and the essential
decomposition of $(S^2,\Pi)$}

\begin{figure}[h!]
\centering
\includegraphics[scale=.355]{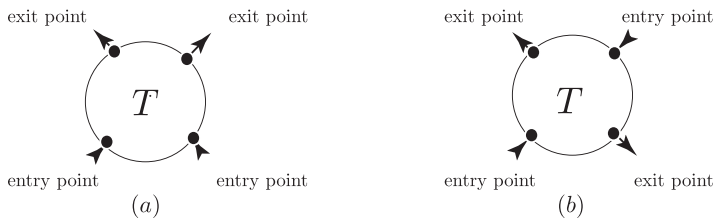}
\caption{Orientation of the boundary points}
\end{figure}

In order to describe the two cases of $Fix (\widetilde \varPhi)$ stated in 
\S 3.1.1 in terms of the essential decomposition of $(S^2,\Pi)$, let us first
describe how the boundary points of a tangle are oriented.

Let $\mathcal{T}$ be a tangle of a projection $\Pi $. The intersection
points of $\partial \mathcal{T}\cap \Pi $ are called the \textbf{boundary
points} of $\mathcal{T}$. By the orientation and the connectivity of $\Pi $,
the four boundary points of $\mathcal{T}$ are oriented such that two are entry
points and both others are exit points (see Fig. 18). Up to a change in the global orientation
of the strands and up to a rotation of angle $\frac{\pi }{2}$, we have the
two possible configurations described in Fig. 18.

\begin{lemma}
Let $\gamma $ be a canonical or essential Conway circle of $\Pi $. If $%
\gamma $ is $\Phi _{\Pi }$-invariant then $q=2$.
\end{lemma}

\begin{proof} 

Note that each orbit of the action of $\Phi _{\Pi }$ has $q$ elements except
the two orbits composed by the fixed points. 

Let $\Delta _{1}$ and $\Delta _{2}$ be the two disks in the projection
sphere such that  $\gamma$ is $\partial \Delta _{1}=\partial \Delta _{2} $. 

The two disks $\Delta _{1}$ and $\Delta _{2}$ 
are either permuted or invariant by $\Phi _{\Pi }$.

Assume $\Phi _{\Pi }$ permutes the two disks $\Delta _{1}$ and $\Delta _{2}$. Since $\Phi _{\Pi }$ preserves $\gamma \cap K$ and there
are no fixed points by $\Phi _{\Pi }$ on $K$, we have that the set $\gamma
\cap K$ is either an orbit of $\Phi _{\Pi }$ with $q=4$ points, or two orbits with each $%
q=2$ points.\\
Assume $q=4$. Let $S_{\gamma}$ be the Conway sphere corresponding to $\gamma $
which is invariant by $\Phi _{\Pi }$. The homeomorphism $\left. \Phi _{\Pi
}\right\vert _{S_{\gamma}}$ is of order two and preserves the orientation of $S_{\gamma}$. Therefore,
by Ker\'{e}kj\'{a}rto theorem, it is topologically conjugate to a rotation of order 2
of $S_{\gamma}$. Thus $\left. \Phi _{\Pi }\right\vert _{S_{\gamma}}$ has two
fixed points, which must also be the fixed points of $\Phi
_{\Pi }$. This implies that $\left. \Phi _{\Pi }\right\vert _{S_{\gamma}}$ preserves the
orientation. Then the two connected components of $S^{3}-S_{\gamma}$ are invariant by 
$\Phi _{\Pi }$, but this contradicts the hypothesis that $\Phi _{\Pi }(\Delta _{1})=\Delta
_{2}$. Therefore the two disks $\Delta _{1}$ and $\Delta _{2}$ are invariant by $\Phi _{\Pi }$. 

$\star$ Assume $\left.\Phi_{\Pi }\right\vert _{\Delta _{1}}$ preserves the orientation. Then  by Ker\'{e}kj%
\'{a}rto's theorem, $\left. \Phi _{\Pi }\right\vert _{\Delta _{1}}$ is
topologically equivalent to a rotation of $\Delta _{1}$. As in the above
case, since $\Phi _{\Pi }$ preserves $\gamma \cap \Pi$ and there are no fixed
points on $\Pi$, the set $\gamma \cap \Pi$ is either an orbit of $%
\Phi _{\Pi }$ with $q=4$ points or two orbits each with $q=2$. \\The case $q=4$ is excluded: since $%
\Phi _{\Pi }(\Delta _{i})=\Delta _{i}$ and $\gamma \cap \Pi$ is an orbit,
 the points in $\gamma \cap \Pi$ would be all entry points or all exit points, but this
is impossible (for the orientation of the boundary points of a tangle of a projection as described above). 

$\star$ $\star$ Assume $\left. \Phi _{\Pi }\right\vert _{\Delta _{1}}$ reverses the
orientation and $q >2$. Then $\left. \Phi _{\Pi }^{2}\right\vert _{\Delta
_{1}}$ preserves the orientation of $\Delta _{1}$ and has period ${q/ 2} > 1$.
Then $\left. \Phi _{\Pi }^{2}\right\vert _{\Delta _{1}}$ is conjugate to a
rotation such that its fixed points are also the fixed points
of $\left. \Phi _{\Pi }\right\vert _{\Delta _{1}}$, but this contradicts the 
hypothesis that $\left. \Phi _{\Pi }\right\vert _{\Delta _{1}}$ reverses
the orientation.

\end {proof}

\begin{proposition}
If $q \geq 3$, there are no edges in $Fix(\widetilde{\varPhi})$.
\end{proposition}

\begin{proof}

Assume there is an edge $E$ of $\widetilde{\mathcal{A}}(K)$ fixed by $%
\widetilde{\varPhi}$. The edge $E$ corresponds to a Conway circle $\gamma
_{E}$ invariant by $\Phi _{\Pi }$. Then the
proposition follows from Lemma 3.1.

\bigskip 

\end {proof}

\begin{corollary}
If $q\geq 3$, $Fix(\widetilde{\varPhi})=V_{0}$ where $V_{0}$ is a vertex of $%
\widetilde{\mathcal{A}}(K)$.
\end{corollary}

\subsection{Proof of Visibility Theorem 3.1 and Applications}

\subsubsection{Proof of Visibility Theorem 3.1}

According to Remark 2.8, if the $q$-periodic knot $K$ is a jewel without
boundary or a torus knot of type $(2,q)$, we are done. \newline
Since the non-torus rational knots are only 2-periodic (see for instance
Theorem 3.1 in \cite{glm}), the hypothesis $q\geq 3$ excludes the case of
rational knots.\newline
All that remains is the case of a projection $\Pi $ whose $\mathcal{C}%
_{ess}(\Pi )$ is not empty. According to Corollary 3.3, the set $Fix(\tilde{%
\varPhi})$ is reduced to a vertex $V_{0}$ representing a jewel or a TBD.
\begin{proof}

\item 
\begin{enumerate}
\item Case where $Fix(\tilde{\varPhi})=V_{0}$ where $V_{0}$ corresponds
to a jewel $J_{0}$ with non-empty boundary. \\
Let $\gamma _{1},\dots ,\gamma _{k}$ be the boundary components of $J_{0}$.
Each essential Conway circle $\gamma _{i}$ bounds on $S^{2}$ a disk $\Delta
_{i}$ which does not meet the interior of $J_{0}$. Consider the tangles $%
\mathcal{T}_{i}=(\Delta _{i},\tau _{\Delta _{i}})$ where $i=1,\dots ,k$.
Hence the $k$ underlying discs are distinct. \newline
Since $J_{0}$ is a jewel, no flypes can occur in $J_{0}$. Since $\Phi _{\Pi }
$ does not leave the edges invariant, there are no invariant boundary circles and by Ker\'{e}kj%
\'{a}rto's theorem applied to $\left. \Phi _{\Pi }\right\vert _{J_{0}}$ is topologically conjugate to a
rotation and has two fixed points in the interior of $J_{0}$. By using a
flat homeomorphism, we can modify $\Pi $ such that $\left. \Phi _{\Pi
}\right\vert _{J_{0}}$ is a rotation and acts freely on the $k$ boundary
components of $J_{0}$. After this modification, we continue to denote the new projection and homeomorphism respectively by $\Pi $ and $\Phi _{\Pi }$.\newline
Each circle $\gamma _{i}$ has $q$ images in its orbit. Thus $k=n q$ and we have $k$ distinct tangles $\mathcal{T}_{i}=(\Delta _{i},\Pi
\cap \Delta _{i})$ with underlying disks $\Delta _{i}$ where $i=1,\dots
 k$.\newline
Note that the $k$ boundary components of $J_{0}$ correspond to the $k$ adjacent vertices to $V_0$

Consider the disk $\Delta _{1}$ and its images $\phi (\Delta _{1}),\dots
,\phi ^{q-1}(\Delta _{1})$ denoted by 
\begin{equation*}
\Delta _{1},\Delta _{2},\dots ,\Delta
_{q-1} 
\end{equation*}
and the corresponding tangles 
\begin{equation*}
(\Delta _{1},\tau _{\Delta _{1}}),(\Delta _{2},\tau _{\Delta _{2}}),\dots
,(\Delta _{q-1},\tau _{\Delta _{q-1}})
\end{equation*}%
where $\tau _{\Delta _{i}}=\Pi \cap \Delta _{i}=\Phi _{\Pi }^{i-1}(\tau
_{\Delta _{1}})$ for $i=1,\dots ,q$.

Consider $\Delta _{j}=\phi ^{j-1}(\Delta _{1})$. Then $\phi ^{j-1}(\tau
_{\Delta _{1}})\subset \Delta _{j}$. Since $\phi (\tau _{\Delta _{1}})$ is
equivalent up to standard flypes to $\tau _{\Delta _{2}}$ and since the disks $%
\Delta _{1}$ and $\Delta _{2}$ are distinct, we can independently modify the
projection $\tau _{\Delta _{2}}=\Pi \cap \Delta _{2}$ by flypes such that $%
\tau _{\Delta _{2}}$ is replaced by $\phi (\tau _{\Delta _{1}})$. With these
modifications, we have a new $\Pi ^{\prime }$ and $\Phi _{\Pi ^{\prime
}}^{\prime }$ such that $\left. \Phi _{\Pi ^{\prime }}^{\prime }\right\vert
_{\Delta _{1}}$ does not need flypes. Hence $\left. \Phi _{\Pi ^{\prime
}}^{\prime }\right\vert _{\Delta _{1}}=\left. \phi ^{\prime }\right\vert
_{\Delta _{1}}$ is flat and $\left. \phi \right\vert _{\Delta
_{1}}^{q}=\left. \Phi _{\Pi }^{\prime }\right\vert _{\Delta _{1}}^{q}=%
\mathrm{Id}$. By repeating this process when necessary in all the orbits, we get a new
alternating projection $\widehat{\Pi }$ of $K$ admitting the symmetry $%
\widehat{\Phi }_{\widehat{\Pi }}$ which is a $q$-rotation whose the two fixed
points are inside $J_{0}$.

\item Case where $Fix(\tilde{\varPhi})=V_{0}$ where $V_{0}$ corresponds
to a TBD $\mathcal{T}_{0}=(\Sigma ,\Sigma \cap \Pi )$. Since $\mathcal{T}_{0}
$ is invariant by $\Phi _{\Pi }$, the number of the visible crossings of $\Sigma
\cap \Pi $ is $m q$. Since there are no edges invariant by $%
\widetilde{\Phi }_{\Pi }$, the boundary components of $\Sigma $ are
distributed in $\Phi _{\Pi }$-orbits of $q$ elements and the total number of
these components is $s q$ for some integer $s \geq 1$. By Ker\'{e}kj\'{a}rto's theorem, $\phi $ is
equivalent to a rotation of order $q$ with the two fixed points in $\Sigma $. 

We first modify the projection $\Pi $ and $\Sigma$ such that: 

- $\Sigma$ is contained in a $2$-sphere invariant by the $q$-rotation $r$,

- the visible crossings of $\mathcal{T}_{0}$ are in $q$ twist regions, such that each
twist region has $s$ crossings and the $q$ twist regions are symmetric with
respect to the rotation $r$ and 

- the boundary components of $\mathcal{T}_{0}$ are $s q$ circles distributed in $s$ orbits of the
action of $r$. 

This results in a new projection $\Pi ^{\prime }$ with the such modified TBD $\mathcal{T}^{\prime }_{0}=(\Sigma^{\prime } ,\Sigma^{\prime } \cap \Pi ^{\prime })$.  Then by a process similar to that described above, we will perform flypes if necessary inside the $s q$ tangles whose boundaries are the boundary components of $\mathcal{T}_{0}$, to obtain a projection $\widehat{\Pi }$ displaying the symmetry of order $q$.

\end{enumerate}
\end{proof}
{\bf Question}: Are there any restrictions on the values of $q$ in Visibility Theorem 3.1?: 

\begin{proposition} 
For prime alternating \textbf{knots} where $Fix(\tilde{\varPhi})=V_{0}$ with $V_{0}$ corresponding to a TBD, only the periods $q \equiv 1 \mod 2$ are possible. 
\end{proposition}
For each tangle of a projection $\Pi$, there are four boundary points located in North-West (NW), North-East (NE), South-East (SE) and South-West (SW). The projection $\Pi$ connects these four points in pairs with three possible connection paths (Fig. 19).\\
If $\Pi$ connects,\\
(1) NW to NE and SW to SE, we have the {\bf H-connection path},\\
(2  NW to SW and NE to SE, we have the {\bf V-connection path},\\
(3) NW to SE and NE to SW, we have  the {\bf  H-connection path}

\begin{figure}[h!]    
   \centering
    \includegraphics[scale=.4]{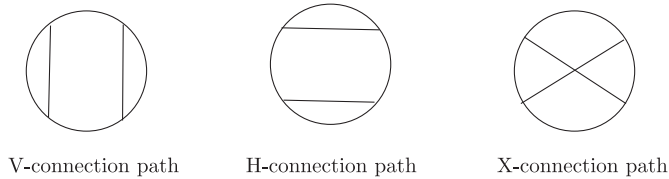} 
\caption{Connection paths of a tangle}       
\end{figure}
\begin{proof} The proof is straightforward by an examination of the possible connection paths (Fig. 19) on the tangles of the TBD. In the case of knots, $q$ cannot be even.\\
 It is interesting to compare this proposition to the result of \cite{boy}. 
\end{proof}
 
\textbf{Conclusion}: Hence for an alternating periodic knot $K$ with period $%
q\geq 3$, there always exists a $q$-periodic alternating projection of $K$.
The only possible obstruction case for a $q$-periodic alternating
projection is when $q=2$. Theorem 3.1 is the equivalent of the Order 4
Theorem 7.1 (\cite{erquwe}) in the study of the visibility of the $+$-
achirality of alternating knots.

\bigskip

\subsubsection{Applications}

(1) Seifert's algorithm applied to a $q$-periodic alternating projection
of a knot $K$ gives rise to a Seifert surface having the genus of $K$ (see
for instance \cite{ga}):

\begin{proposition}
There exists a $q$-equivariant orientable surface of $K$ with minimal genus
for $q\geq 3$.
\end{proposition}

(2) From Visibility Theorem 3.1, we have:

\begin{proposition}
The crossing number of a prime alternating knot that is $q$-periodic with $%
q \geq 3$ is a multiple of $q$.
\end{proposition}

(3) We now use Visibility Theorem 3.1 with the Murasugi decomposition of
alternating links (\cite{quwe1},\cite{quwe2}) to study the 3-periodicity of
the knot $12a_{634}$. We have

\begin{proposition}
The knot $12a_{634}$ is not 3-periodic.
\end{proposition}

\begin{figure}[h!]
\centering
\includegraphics[scale=.3]{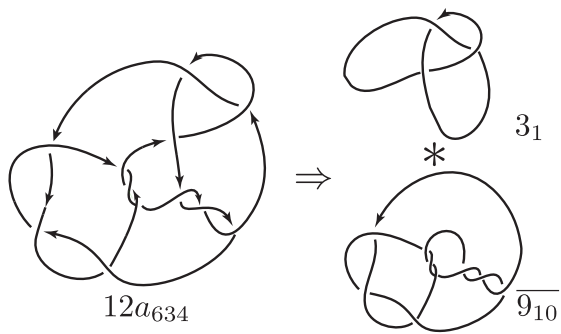}
\caption{$12a_{634}=3_1 *\overline{9_{10}} $}
\end{figure}
\begin{figure}[h!]
\centering
\includegraphics[scale=.4]{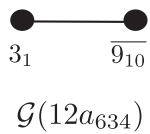}
\caption{ Adjacency graph $\mathcal{G}(12a_{634})$}
\end{figure}
\begin{proof}
Let us consider the Murasugi decomposition of the knot $12a_{634}$ (Fig. 20) and its adjacency graph $\mathcal G(12a_{634})$. With the notations of \cite{quwe2}, we have the Murasugi decomposition of $12a_{634}$ as:
$$12a_{634} = 3_1 *\overline{9_{10}}$$
where the knot $\overline{9_{10}}$ is the mirror image of ${9_{10}}$.
Thus, the adjacency graph $\mathcal G(12a_{634})$ is a tree with 2 vertices corresponding to the trefoil knot $3_{1}$ and the knot $\overline{9_{10}}$ (Fig. 21).

The following lemma is useful for our analysis:

\begin{lemma} Suppose that a prime non-splittable oriented link $L$ has a $q$-periodic alternating diagram and that its Murasugi adjacency graph is a tree with 2 vertices. Then the two constituent atoms of $L$ are $q$-periodic.

\end{lemma}

\begin{proof} According to Corollary 1 (in \cite{co}), since the adjacent graph ${\cal G}(L)$ is a tree,  its periodic automorphism has a fixed point which corresponds to a $q$-periodic atom. Moreover in the case where ${\cal G}(L)$ has only two vertices, the periodic automorphism is reduced to the identity and the two atoms are therefore both $q$-periodic.
\end{proof}

By Theorem 3.1, if the knot $12a_{634}$ were 3-periodic, it would admit a 3-periodic alternating projection and we would be able to apply Lemma 3.2. 
However since the knot $\overline{9_{10}}$, one of the two constituent atoms of $12a_{634}$ is a non-torus rational knot, hence it is not 3-periodic; it is only 2-periodic (Theorem 6.1 in \cite{glm}). Hence by Lemma 3.2, we can conclude that $12a_{634}$ is not 3-periodic.
 \end {proof}
 
With this result, like S. Jabuka and S. Naik \cite{jana}, we thus complete the tabulation of the $q$-periodic prime alternating twelve-crossing knots where $q$ is an odd prime but our proof is not supported on computer calculations.

\begin{remark}
The Murasugi decomposition with its adjacency graph enables to deduce that
the knot $12a_{634}$ is not $q$-periodic for any $q \geq 3$, chiral and
non-invertible (\cite{quwe1}, \cite{quwe2}).
\end{remark}

\section{Addendum}

There are overlapping results between the paper of Keegan Boyle \cite{boy}
and this one. Both these papers use flypes as a main tool, but differ in
their techniques. Since non-trivial flypes lie completely in the arborescent
part of alternating projections and since the canonical structure tree
inherits the $q$-periodicity of a $q$-periodic alternating knot, we can
adjust by flypes to derive a $q$-periodic alternating projection. Therefore
our proof is somewhat constructive and also deals with the case of $q$%
-periodicity with $q$ even.

\end{document}